\documentclass[11pt,reqno]{amsart}
\usepackage[colorlinks=true,allcolors=blue,backref=page]{hyperref}
\usepackage{color}
\usepackage{amsmath, amssymb, amsthm}

\usepackage{mathrsfs}
\usepackage{mathtools}
\usepackage[noabbrev,capitalize,nameinlink]{cleveref}
\crefname{equation}{}{}
\usepackage{fullpage}
\usepackage[noadjust]{cite}
\usepackage{graphics}
\usepackage{pifont}
\usepackage{tikz}
\usepackage{tikz-cd}
\usepackage{bbm}
\usepackage[T1]{fontenc}

\usetikzlibrary{arrows.meta}

\usepackage{environ}
\usepackage{framed}
\usepackage{url}
\usepackage[linesnumbered,ruled,vlined]{algorithm2e}
\usepackage[noend]{algpseudocode}
\usepackage[labelfont=bf]{caption}
\usepackage{cite}
\usepackage{framed}
\usepackage[framemethod=tikz]{mdframed}
\usepackage{appendix}
\usepackage{graphicx}
\usepackage[textsize=tiny]{todonotes}
\usepackage{tcolorbox}
\usepackage{enumerate}
\usepackage[shortlabels]{enumitem}
\allowdisplaybreaks[1]

\apptocmd{\sloppy}{\hbadness 10000\relax}{}{} 

\crefname{algocf}{Algorithm}{Algorithms}

\crefname{equation}{}{} 
\crefname{conjecture}{Conjecture}{Conjectures} 
\AtBeginEnvironment{appendices}{\crefalias{section}{appendix}} 

\crefformat{enumi}{#2#1#3}
\crefrangeformat{enumi}{#3#1#4 to~#5#2#6}
\crefmultiformat{enumi}{#2#1#3}%
{ and~#2#1#3}{, #2#1#3}{ and~#2#1#3}

\usepackage[final]{showkeys} 

\colorlet{refkey}{orange!20}
\colorlet{labelkey}{blue!30}

\crefname{algocf}{Algorithm}{Algorithms}

\numberwithin{equation}{section}
\newtheorem{theorem}{Theorem}[section]
\newtheorem{proposition}[theorem]{Proposition}
\newtheorem{lemma}[theorem]{Lemma}

\crefname{claim}{Claim}{Claims}

\newtheorem*{question*}{Question}

\theoremstyle{definition}
\newtheorem{definition}[theorem]{Definition}

\newtheorem*{definition*}{Definition}

\theoremstyle{remark}
\newtheorem*{remark}{Remark}

\newcommand{\eps}{\varepsilon}

\newcommand\poly{\operatorname{poly}}

\renewcommand{\le}{\leqslant}
\renewcommand{\ge}{\geqslant}

\newcommand\Z{\mathbf{Z}}
\newcommand\Q{\mathbf{Q}}
\newcommand\C{\mathbf{C}}
\newcommand\R{\mathbf{R}}
\newcommand\N{\mathbf{N}}

\newcommand\id{\operatorname{id}}

\allowdisplaybreaks

\newcommand{\md}[1]{\ensuremath{(\operatorname{mod}\, #1)}}

\usepackage{graphicx}

\title{On Alweiss's example for multiple recurrence}
\author[A1]{Ben Green}
\address{Mathematical Institute, Andrew Wiles Building, Radcliffe Observatory Quarter, Woodstock Rd, Oxford OX2 6QW, UK}
\email{ben.green@maths.ox.ac.uk}

\begin{document}

\begin{abstract}
In this expository note we discuss the example of R.~Alweiss giving a construction of a set $S \subset \N$ which intersects every nil-Bohr set, but which is not a set of $2$-recurrence. 
\end{abstract}

\maketitle
\setcounter{tocdepth}{1}

\section{Introduction}

In this note, each nilmanifold $G/\Gamma$ is assumed to be endowed with a metric $d_{G/\Gamma}(x,y) = \inf_{\gamma \in \Gamma} d_G(x, y \gamma)$ for some right-invariant metric $d_G$ on the simply-connected nilpotent group $G$. Such a metric $d_G$ can be described in some specified way with reference to a Mal'cev basis for the Lie algebra $\mathfrak{g}$, which we assume to have rational structure constants. (See \cite{GT-nil} for details of one construction.) By a \emph{nil-Bohr set} we mean a set of the form 
\begin{equation}\label{nil-Bohr} B = \{ n \in \N : d_{G/\Gamma}(g^n, \id_G) \le \eps\}\end{equation} for some $\eps > 0$, where $\id_G$ denotes the identity of $G$. Our aim is to describe a construction of Alweiss \cite{alweiss}, which establishes the following theorem. 

\begin{theorem}\label{mainthm}
There exists a set $S \subset \N$ such that the following are true:
\begin{enumerate}
    \item $S$ intersects every nil-Bohr set;
    \item There is a finite colouring of $\N$ such that there is no monochromatic 3-term progression $x, x+d, x+2d$ with $d \in S$.
\end{enumerate}
\end{theorem}
\begin{remark} The paper \cite{alweiss} passes to an equivalent notion of nil-Bohr set defined using bracket polynomials. Working with nilmanifolds throughout makes the main ideas of the argument clearer and avoids any appeal to the literature connecting nilmanifolds and bracket polynomials.  \end{remark}

\emph{Acknowledgements and AI use.} I thank Dan Altman and Mehtaab Sawhney for helpful remarks and Ryan Alweiss for correspondence about his paper. 

In the main part of the note, only very minor use was made of AI tools. However, I used ChatGPT 5.6 Pro to prove \cref{prop41}, which shows that the set $S$ \emph{is} a set of recurrence.\vspace*{8pt}

\emph{Notation.} We write $\Vert x \Vert$ for the distance from $x$ to the nearest integer. Set $\delta := \frac{1}{200}$.

\section{The construction and the colouring}

We now describe the construction of $S$ and give the colouring satisfying \cref{mainthm}.  We will choose real numbers $1 > \alpha_1 > \alpha_2 > \cdots > 0$ and positive integers $N_1 < N_2 < \cdots$. They will be chosen in the order $\alpha_1, N_1, \alpha_2, N_2,\dots$, making sure that $1/\alpha_i$ or $N_i$ is sufficiently large with respect to preceding parameters. The precise constraints necessary will be specified near the beginning of \cref{section3}. We will certainly choose $\alpha_{i+1} < \frac{1}{2}\alpha_i$, and so $\sum_{i = 1}^{\infty} \alpha_i$ converges. 
\begin{definition}\label{s-def} Define $S$ to be the set of all $n \in \N$ such that for some (unique) value of $i$ we have
\begin{enumerate}
    \item $\Vert \alpha_i n - \frac{1}{2} \Vert \le \delta$;
    \item $\Vert \alpha_j n \Vert \le \delta$ for $j \ne i$;
    \item $\sum_{j \ne i} \Vert \alpha_j n \Vert < 4$. 
\end{enumerate}
\end{definition}
Now we define the colouring. Set $f(n) := \sum_{i = 1}^{\infty} ( e(\alpha_i n) - 1) \in \C$. Note that due to the inequality $|e(x) - 1| \le 2\pi |x|$ and the convergence of $\sum_{i=1}^{\infty} \alpha_i$, $f(n)$ is well-defined for every $n$.

Set $K := 8$ and divide the torus $\C/8\Z[i]$ into $64K^2$ squares of sidelength $1/K$ in the obvious way, and colour $\N$ according to which square $f(n) \md{8\Z[i]}$ lies in. 
Suppose, as a hypothesis for contradiction, that we have a monochromatic 3-term arithmetic progression $x, x+d, x+2d$ with $d \in S$. Let $i$ be the integer associated to $d$ as in \cref{s-def}. We have
\begin{align}\nonumber f(x) - 2f(x + d) + f(x + 2d) & = e(\alpha_i x) - 2 e(\alpha_i(x + d)) + e(\alpha_i(x + 2d)) \\ & + \sum_{j \ne i} \big( e(\alpha_j x) - 2 e(\alpha_j(x + d)) + e(\alpha_j(x + 2d)) \big). \label{f-expand}\end{align} 
The first term here is a complex number $z$ with $2 \le |z| \le 4$. The upper bound is obvious; to see the lower bound, note that  
\[ |z| =\big| 1 - 2 e(\alpha_i d) + e(2 \alpha_i d) \big|  =  4 \sin^2 (\pi \alpha_i d) \ge 4 \cos^2 (\pi \delta) \ge 2,\] where in the penultimate step we used \cref{s-def} (1). The sum over $j \ne i$ in \cref{f-expand} is a complex number $w$ with 
\[ |w| \le 4 \sum_{j \ne i} \sin^2 (\pi \alpha_j d) \le 4\pi^2 \sum_{j \ne i} \Vert \alpha_j d \Vert^2 \le 16\pi^2 \delta \le 1, \] where in the last step we used \cref{s-def} (2) and (3) (and the fact that $\delta = \frac{1}{200}$). It follows from these observations, \cref{f-expand}, and the triangle inequality that 
\begin{equation}\label{sandwich-f} 1 \le \big| f(x) - 2f(x + d) + f(x + 2d)\big| \le 5. \end{equation}
However, since $x, x+d, x+2d$ are all the same colour, $f(x) - 2f(x +d) + f(x + 2d)$ is at (Euclidean) distance at most $4\sqrt{2}/K < 1$ from the lattice $8\Z[i]$. This is incompatible with \cref{sandwich-f} and so we have a contradiction.

\section{$S$ is a set of nil-Bohr recurrence}\label{section3}

We now turn to the proof that $S$ is a set of nil-Bohr recurrence, or in other words that \cref{mainthm} (1) holds true. The set of all nilmanifolds $G/\Gamma$ (together with metric $d_{G/\Gamma}$) is countable, being determined by dimension, step and the rational structure constants of the Mal'cev basis. List them in some arbitrary fashion and denote by $r(G/\Gamma)$ the index of a given $G/\Gamma$ in this listing. 

Recall the notion of a \emph{polynomial sequence} $(g(n))_{n \in \N}$ adapted to the lower central series filtration on $G$ (see for instance \cite[Definition 1.8]{GT-nil}). Each `linear' sequence $(g^n)_{n \in \N}$ is an example of a polynomial sequence. We will need two facts from nilmanifold theory. The first is the following compactness property of the polynomial sequences relative to a nilmanifold. 

\begin{lemma}\label{poly-compact}
Suppose that $G/\Gamma$ is an $s$-step nilmanifold and that $g(n)$ is a polynomial sequence on $G$ with $g(0) = \id_G$. Then $g(n)$ has a unique Taylor expansion $g(n) =   g_1^{\binom{n}{1}}\cdots g_s^{\binom{n}{s}}$ with $g_i \in G_i$, the $i$th group in the lower central series of $G$. Moreover there exists $\tilde g(n)$ such that $\tilde g(n)\Gamma = g(n) \Gamma$ for all $n$, and such that the Taylor coefficients $\tilde g_i$ all satisfy $d_{G}(\tilde g_i, \id_G) = O_{G/\Gamma}(1)$, uniformly in $g$.
\end{lemma}
\begin{proof} For the existence of the Taylor expansion, see \cite[Lemma 6.7]{GT-nil}. For the boundedness statement see, for example, \cite[Lemma C.1]{GTZ-uk}. The proof is by iteratively passing to $\Gamma$-equivalent polynomial sequences with bounded $g_1, g_2,\cdots$ in this order. In the 2-step case we can be quite explicit via use of the relation $(ab)^n = a^n b^n [a,b]^{-\binom{n}{2}}$; the first step of the process is then to replace $g_1^n g_2^{\binom{n}{2}}$ by $\tilde g_1^n (g'_2)^{\binom{n}{2}}$ where $g'_2 := g_2 [\tilde g_1, \gamma_1]^{-1}$, and $g_1 = \tilde g_1 \gamma_1$ with $\tilde g_1$ bounded and $\gamma_1 \in \Gamma$.\end{proof}

The Taylor expansion allows us to define a metric on the space of polynomial sequences on a given $G$, defined by $d_{\poly, G}(g, g') := \max_i d_G(g_i, g'_i)$. Each evaluation map $g \mapsto g(n)$ is continuous in this metric.

The second result we need is the following theorem of Leibman \cite{leibman-2}.

\begin{proposition}\label{leibman-prop} Suppose that $(g(n))_{n \in \N}$ is a polynomial sequence in $G$ with $g(0) = \id_G$. Then we have a factorization of polynomial sequences $g(n) = g'(n) \gamma(n)$, where $(g'(n)\Gamma)_{n \in \N}$ is totally equidistributed in some subnilmanifold $G'/\Gamma'$, and $(\gamma(n)\Gamma)_{n \in \N}$ is periodic with $\gamma(0) = \id_G$. In particular, there is an infinite subprogression $P = \{d, 2d,\cdots\}$ such that $\overline{(g(n)\Gamma)_{n \in P}} = G'/\Gamma'$.
\end{proposition}
\begin{proof}
See \cite[Theorem B]{leibman-2} or, for precisely the statement we have given here, \cite[Corollary 1.12]{GT-nil}.
\end{proof}

The following is the key technical result required for the proof that $S$ is a set of nil-Bohr recurrence.

\begin{proposition} \label{prop32} Suppose that $\alpha_1,\dots, \alpha_m$ have already been chosen. Then there is a choice of $N_m$ such that the following is true uniformly for all $G/\Gamma$ with $r(G/\Gamma) \le m$ and all polynomial sequences $g(n)$ on $G$: $(n\alpha_1,\dots, n\alpha_m, g(n))_{n \le N_m}$ is $(1/m)$-dense\footnote{If $X$ and $Y$ are subsets of a metric space then we say that $X$ is $\eps$-dense in $Y$ if, for every $y \in Y$, there is $x \in X$ with $d(x,y) \le \eps$. We do not require $X \subseteq Y$.} on some subnilmanifold $G'/\Gamma' \le \R^m/\Z^m \times G/\Gamma$ with $\dim G' \ge m$. \end{proposition}
\begin{remark} $G'/\Gamma'$ is allowed to depend on $g$. The metric on $\R^m/\Z^m \times G/\Gamma$ is the (sup) product metric of the $\ell^{\infty}$ metric on $\R^m/\Z^m$ and $d_{G/\Gamma}$.\end{remark}
\begin{proof}
By taking a supremum of the resulting $N$'s, it suffices to prove the result for a single fixed nilmanifold $G/\Gamma$. By \cref{poly-compact}, it is enough to prove the result for polynomial sequences $g(n)$ with $d_{\poly, G}(g,\iota) \le C$ for some $C = O_{G/\Gamma}(1)$, where $\iota$ denotes the trivial polynomial sequence with $\iota(n) = \id_G$ for all $n$.

By \cref{leibman-prop}, for each $g$ the orbit $(n\alpha_1,\dots, n\alpha_m, g(n))_{n \in \N}$ is dense in some nilmanifold $G'_g/\Gamma'_g \le \R^m/\Z^m \times G/\Gamma$. Since $1,\alpha_1,\dots,\alpha_m$ are independent over $\Q$, the projection of $G'_g$ to $\R^m$ is surjective, and so $\dim G'_g \ge m$. Choose some finite $N(g)$ such that $(n\alpha_1,\dots, n\alpha_m, g(n))_{n \in N(g)}$ is $(1/2m)$-dense on $G'_g/\Gamma'_g$.

By continuity of the evaluation maps $g \mapsto g(n)$, for each $g$ there is some open neighbourhood $B_g$ of $g$ (in the metric $d_{\poly, G}$ on the space of polynomial sequences) such that $(n\alpha_1,\dots, n\alpha_m, \tilde g(n))_{n \le N(g)}$ is $(1/m)$-dense on $G'_g/\Gamma'_g$ for all $\tilde g \in B_g$. 

Since the ball $\{ g : d_{\poly, G}(g,\iota) \le C\}$ is compact, there is a finite subcover $(B_g)_{g \in X}$. Setting $N := \max_{g \in X} N(g)$, we then see that $(n\alpha_1,\dots, n\alpha_m, g(n))_{n \le N}$ is $(1/m)$-dense in some $G'/\Gamma'$ for every $g$.
\end{proof}

We may now fully describe the choice of the parameters $\alpha_i$ and $N_i$. We initiate the construction by choosing $\alpha_1$ to be some arbitrary element of $[0,1] \setminus \Q$, say $\alpha_1 = 1/\sqrt{2}$. The construction will be done in such a way that $1, \alpha_1,\alpha_2,\dots,\alpha_m$ are independent over $\Q$ and $\alpha_{i+1} < \alpha_i/2$ for all $i$ (in practice $\alpha_{i+1}$ will be vastly smaller). Supposing that $\alpha_1,N_1,\alpha_2,\dots, N_{m-1},\alpha_m$ have already been chosen, we choose $N_m$ as in \cref{prop32}. Then, supposing that $\alpha_1,N_1,\alpha_2,N_2,\dots, \alpha_m, N_m$ have been constructed, choose $\alpha_{m+1}$ so that 
\begin{equation}\label{const-3} \alpha_{m+1} < \frac{\delta}{10} N_m^{-1}.   \end{equation}

Now we turn to the main proof of \cref{mainthm} (1). Fix a nil-Bohr set $B$ as in \cref{nil-Bohr} and set $k := \dim G$, which we assume positive. Without loss of generality we may assume $\eps \le \min(\delta/10, 1/10k)$. Let $m \ge \max(r(G/\Gamma), (10/\delta\eps)^{k+1})$, so in particular $m > 2/\eps$ and $m > k$. Consider the orbit $\mathcal{O} = (n\alpha_1,\dots,n\alpha_m, g^n)_{n \le N_m}$ in $\R^m/\Z^m \times G/\Gamma$. By \cref{prop32}, this is $(1/m)$-dense on some subnilmanifold $G'/\Gamma'$ with $\dim G' \ge m$. Let $H \le \R^m$ be such that $H \times \{\id_G\} = G' \cap (\R^m \times \{\id_G\})$; thus $H$ is simply a vector subspace of $\R^m$. Since $\dim G' \ge m$, $\dim H \ge m - k$. Pick a set $I \subset [m]$ of $k$ indices such that $H + \R^I = \R^m$.

Denote by $e_1,\dots, e_m$ the standard basis vectors of $\R^m$. For each $j \in [m]$ write $\frac{\delta}{2} e_j = h_j + x_j$ where $h_j \in H$ and $x_j \in \R^I \subset \R^m$ (where $\delta = \frac{1}{200}$). By pigeonhole and the lower bound on $m$, there exists $i \in [m]$ and a set $J \subset [m] \setminus \{i\}$, $|J| = 1/\delta$, such that $\Vert x_i - x_j \Vert_{\ell^{\infty}(\R^I/\Z^I)} \le  \delta\eps/2$ for $j \in J$. Then 
\begin{equation}\label{h-def} \frac{1}{2} e_i - \sum_{j \in J} \frac{\delta}{2} e_j  = \sum_{j \in J} (\frac{\delta}{2} e_i - \frac{\delta}{2} e_j) =  h + x\end{equation} with $h = \sum_{j \in J} (h_i - h_j) \in H$ and $x = \sum_{j \in J} (x_i - x_j) \in \R^I$ satisfying $\Vert x \Vert_{\ell^{\infty}(\R^I/\Z^I)} \le \eps/2$.
Now $\mathcal{O}$ is $(1/m)$-dense in $G'/\Gamma'$, so there is $n \le N_m$ such that 
\[ d_{\R^m/\Z^m \times G/\Gamma}\big( (h, \id_G),(\alpha_1 n,\dots, \alpha_m n, g^n) \big) \le 1/m. \]
It follows from the definition \cref{h-def} of $h$ that 
\begin{equation}\label{h-alphas} \Vert \frac{1}{2} e_i - \sum_{j \in J} \frac{\delta}{2} e_j - x - (\alpha_1 n,\dots, \alpha_m n) \Vert_{\infty} \le 1/m\end{equation}
and also that $d_{G/\Gamma}(g^n, \id_{G}) \le 1/m$. Since $m > 1/\eps$, this latter fact implies that $n \in B$.

We will show that \cref{h-alphas} implies that $n \in S$, which then concludes the proof that $S \cap B$ is nonempty. Recall that $x$ is supported on $\R^I$ and that $\Vert x \Vert_{\ell^{\infty}(\R^I/\Z^I)} \le \eps/2$. Therefore from \cref{h-alphas} (and since $m > 2/\eps$ and $\eps < \delta/10$) we have
\begin{equation}\label{first-bd-1} \Vert \alpha_i n - \frac{1}{2}\Vert \le \frac{\eps}{2} + \frac{1}{m} < \eps < \delta.\end{equation} For $j \in J$, we have
\begin{equation}\label{first-bd-2} \Vert \alpha_j n \Vert < \frac{\delta}{2} + \frac{\eps}{2} + \frac{1}{m} < \delta.\end{equation} 
For $u \in I \setminus (J \cup \{i\})$ we have
\begin{equation}\label{first-bd-3} \Vert \alpha_u n \Vert \le \frac{\eps}{2} + \frac{1}{m} \le \eps < \delta.\end{equation}
Finally for all other $u \in [m]$ we have
\begin{equation}\label{first-bd-4} \Vert \alpha_u n \Vert \le \frac{1}{m} < \delta.\end{equation}
The inequality \cref{first-bd-1} is already item (1) in \cref{s-def}. For item (2), note that this follows immediately from \cref{first-bd-2,first-bd-3,first-bd-4} for $j \in [m] \setminus \{i\}$. For $j > m$ we have
\begin{equation}\label{first-bd-5} \Vert \alpha_j n \Vert \le \alpha_{m+1} N_m < \delta\end{equation} by the choice of $\alpha_{m+1}$ (see \cref{const-3}).

Finally we verify item (3) in \cref{s-def}. For this, we use \cref{first-bd-2}, the second inequality in \cref{first-bd-3}, and the first inequality in \cref{first-bd-4} to give
\begin{equation}\label{small-j-2} \sum_{j \le m, j \ne i} \Vert \alpha_j n \Vert \le |J|\delta + k \eps + m \frac{1}{m} \le 3,\end{equation} 
where here we used $|J| = 1/\delta$ and $\eps < \frac{1}{10k}$. For the remaining sum over $j > m$, we upgrade \cref{first-bd-5} to $\Vert \alpha_{m + u} n\Vert \le 2^{1 - u} \delta$ for $u \in \N$, which follows in the same way via use of the fact that $\alpha_{j+1} < \alpha_j/2$. Summing over $j$ gives $\sum_{j > m} \Vert \alpha_j n \Vert \le 2 \delta < \frac{1}{2}$. Combining this with \cref{small-j-2} gives the desired condition (3) in \cref{s-def}, and the proof is complete.
\appendix 
\section{$S$ is a set of recurrence}

It is natural to ask whether the set $S$ defined in \cref{s-def} is a set of (topological) recurrence. Since (a famous question of Katznelson) no example is known of a set of Bohr recurrence which is not also a set of recurrence, it is perhaps unsurprising that the answer is affirmative. ChatGPT 5.6 Pro was able to give the following argument after 67 minutes of thought. It was rewritten, with an appropriate reference supplied, by me.

\begin{proposition}\label{prop41}
Suppose the growth of the sequence $\alpha_1^{-1} < N_1 < \alpha_2^{-1} < N_2 < \cdots$ is sufficiently rapid. Then $S$ is a set of recurrence.
\end{proposition}
\begin{proof}
Recall that the \emph{shift graph} $G_{m,k}$ is defined as follows. The vertices of $G_{m,k}$ are tuples $(a_1,\dots, a_k)$ with $1 \le a_1 < \cdots < a_k \le m$, and the edges are between vertices $(a_1,\dots, a_k)$ and $(a_2,\dots, a_{k+1})$. It was shown by Erd\H{o}s and Hajnal \cite{erdos-hajnal} that the chromatic number $\chi(G_{m,k})$ tends to infinity as $m \rightarrow \infty$, for fixed $k$.

Set $k := \lceil 10/\delta\rceil$. Given a vertex $a = (a_1,\dots, a_k) \in G_{m,k}$, first associate the element $x(a) \in \R^m/\Z^m$ defined by $x(a)_{a_j} = j/2k$, $j = 1,\dots, k$ and all other coordinates of $x(a)$ are zero. Next, pick some $n(a) \in [N_m]$ such that the $\ell^1$ distance between $n(a) \cdot (\alpha_1,\dots,\alpha_m)$ and $x(a)$ (on the torus) is at most $\delta^2$. By Kronecker's theorem, this will be possible if $N_m$ is large enough.

Now suppose that $\N$ is $r$-coloured. Pick $m$ is large enough that $\chi(G_{m,k}) > r$, and consider the induced colouring on $G_{m,k}$ given by pulling back the colouring of $\N$ under the map $a \mapsto n(a)$. Pick $a = (a_1,\dots, a_k)$ and $a' = (a_2,\cdots, a_{k+1})$ with the same colour, and set $n := |n(a') - n(a)|$. We claim that $n \in S$, which will establish the proposition.

To see this, observe that $n \cdot(\alpha_1,\dots, \alpha_m)$ is within $2\delta^2$ (in $\ell^1$ of $\R^m/\Z^m$) of $z := x(a') - x(a)$. By construction, $z$ has $a_{k+1}$-coordinate equal to $\frac{k+1}{2k}$, all $a_j$-coordinates of size $\le 1/2k$ ($j = 1,\dots, k$), and all other coordinates zero. We then have
\[ \Vert \alpha_{a_{k+1}} n - \frac{1}{2} \Vert \le 2\delta^2 + \frac{1}{2k} < \delta,\]
and \[ \Vert \alpha_{a_j} n \Vert \le 2\delta^2 + \frac{1}{2k} < \delta\] for $j \in \{1,\dots, k\}$. Moreover
\[ \sum_{u \notin \{a_1,\dots, a_{k+1}\}} \Vert \alpha_u n \Vert \le 2\delta^2,\] and so from the preceding inequalities we have
\[ \sum_{u = 1}^m \Vert \alpha_u n\Vert \le \frac{1}{2} + \delta + k (2\delta^2 + \frac{1}{2k}) + 2\delta^2 < 2.\] 
Since $n \le N_m$, if $\alpha_{m+1},\alpha_{m+2},\dots$ are sufficiently small then it follows that 
\[ \sum_{u = 1}^{\infty} \Vert \alpha_u n \Vert < 4.\]
We have verified the required properties (1), (2) and (3) of \cref{s-def}.
\end{proof}

After a couple of days thought I was not able to decide whether or not $S$ is a set of \emph{measurable} recurrence. I also did not succeed in prompting an LLM to answer this question, which therefore remains open.

\bibliographystyle{amsplain0}
\bibliography{main.bib}

\end{document}